\documentclass[10pt,a4paper]{amsart}
\usepackage[english]{babel}
\usepackage{amssymb,xypic,amscd,amsthm}
\CompileMatrices
\newtheorem{defn}{Definition}
\newtheorem{thm}[defn]{Theorem}
\newtheorem{cor}[defn]{Corollary}
\newtheorem{lem}[defn]{Lemma}
\newtheorem{prop}[defn]{Proposition}

\theoremstyle{plain}
\newtheorem{rem}[defn]{Remark}
\theoremstyle{remark}
\newtheorem{exam}{Example}
\numberwithin{equation}{section} \numberwithin{defn}{section}

\newcommand\ed{\operatorname{End}}

\newcommand\Ker{\operatorname{Ker}}
\newcommand\aut{\operatorname{Aut}}
\newcommand\tr{\operatorname{Tr}}

\newcommand\Det{\operatorname{det}}

\begin{document}

\title[Bounded Finite Potent Operators]{On Bounded Finite Potent Operators \\ on arbitrary Hilbert Spaces}
\author{Fernando Pablos Romo }

\address{Departamento de Matem\'aticas  and Instituto Universitario de F\'{\i}sica Fundamental y Matem\'aticas, Universidad de Salamanca, Plaza de la Merced 1-4, 37008 Salamanca, Espa\~na}
 \email{ fpablos@usal.es}

\thanks{
\\ This work is partially supported  by the
Spanish Government research projects no. PGC2018-099599-B-I00 and the Regional Government of Castile and Leon research project no. J416/463AC03.}

\begin{abstract} The aim of this work is to study the structure of bounded finite potent endomorphisms on Hilbert spaces. In particular, for these operators, an answer to the Invariant Subspace Problem is given and the main properties of its adjoint operator are offered. Moreover, for every bounded finite potent endomorphism we show that Tate's trace coincides with the Leray trace and with the trace defined by R. Elliott for Riesz Trace Class operators.
\end{abstract}

\maketitle

\bigskip

\setcounter{tocdepth}1


\section{Introduction} \label{s:introduc-33736}

The notion of finite potent endomorphism on an arbitrary vector space was introduced by J. Tate in \cite{Ta} as a basic tool for his elegant definition of Abstract Residues.

During the last decade the theory of finite potent endomorphisms have been applied to studying different topics related to Algebra, Arithmetic and Algebraic Geometry. Thus,  A. Yekutieli in \cite{Ye-LBT} and O. Braunling in \cite{Br-A} and \cite{Br-RS} have addressed problems of arithmetic symbols by using properties of finite potent endomorphism; C. P. Debry in \cite{Deb} and L. Taelman in \cite{Tae} have offered results about Drinfeld modules from these linear operators and V. Cabezas S\'anchez and the author of this work have given explicit solutions of infinite linear systems from reflexive generalized inverses of finite potent endomorphisms in \cite{CP2}.

As far as we know a study of finite potent endomorphisms in the context of the Functional Analysis is not stated explicitly in the literature.

The aim of this work is to study the main properties of bounded finite potent endomorphisms on arbitrary Hilbert spaces. Indeed, for these operators, an answer to the Invariant Subspace Problem is given and the main properties of its adjoint operator are offered. Moreover, for every bounded finite potent endomorphism we show that Tate's trace coincides with the Leray trace and with the trace defined by R. Elliott for Riesz Trace Class operators. Also, we relate the determinant of a finite potent endomorphism offered in \cite{HP} with classical determinants defined with techniques of Functional Analysis for trace class operators. Bounded finite rank operators and bounded nilpotent linear maps are particular cases of bounded finite potent endomorphism.

The paper is organized as follows. In section \ref{s:pre} we recall the basic definitions of this work (inner product spaces, Hilbert spaces, bounded operators, orthogonality and the adjoint of a bounded linear map) and a summary of statements of the articles \cite{AST}, \cite{Pa-CN} and \cite{Ta}.

Section \ref{s:345-boun-fp} deals with the study of the main properties of bounded finite potent endomorphisms on Hilbert spaces. Accordingly, the characterization of these ope\-rators is given in Theorem \ref{th:char-boun-fp30363}, the Invariant Subspace Problem is solved for them in Proposition \ref{prop:inv-sub-probl-hil-3094734}  and  Theorem \ref{th:Ris-Trac-Clas-boun-3843353} shows that every bounded finite potent endomorphism on a Hilbert space is a Riesz trace class operator. Moreover, we study the spectrum of bounded finite potent endomorphism, we determine when they are compact, we prove that different definitions of traces on infinite-dimensional Hilbert spaces coincide and we relate the determinant of a finite potent endomorphism with the determinants offered  by N. Dunford and J. Schwartz in \cite{DS} and by B. Simon in \cite{Si} for trace class operators in separable Hilbert spaces.

Finally, Section \ref{s:adj-finite-potent-397} is devoted to offer the characterization of the adjoint $\varphi^*$ of a bounded finite potent $\varphi \in \ed_{\mathbb C} (\mathcal H)$ from the AST-decomposition of $\mathcal H$ introduced in \cite{AST}, the CN-decomposition of $\varphi^*$ given in \cite{Pa-CN}, the structure of the spectrum of $\varphi^*$ and the relation between the trace of $\varphi^*$ and the determinant of $\text{Id} + \varphi^*$ with the trace of $\varphi$ and the determinant of $\text{Id} + \varphi$ respectively.

We hope that from the general properties of bounded finite potent endomorphisms introduced in this work, different applications can be found in the near future.

\bigskip
\section{Preliminaries} \label{s:pre}

This section is added for the sake of completeness.

\subsection{Operators on Hilbert Spaces} \label{ss:inner-product}

Let $k$ be the field of the real numbers or the field of the complex numbers, and let $V$ be a $k$-vector space.

An inner product on $V$ is a map $g \colon V \times V \to k$ satisfying:

\begin{itemize}

\item $g$ is linear in its first argument: $$g(\lambda v_1 + \mu v_2, v') = \lambda g(v_1,v') + \mu g(v_2,v')  \text{ for every }  v_1,v_2, v' \in V\, ;$$

\item $g (v', v) = {\overline {g (v,v')}}$ for all $v,v' \in V$, where ${\overline {g (v,v')}}$ is the complex conjugate of $g(v,v')$;

\item  $g$ is positive definite: $$g(v,v) \geq 0 \text{ and } g(v,v) = 0 \Longleftrightarrow v = 0\, .$$

\end{itemize}

 Note that $g(v,v) \in \mathbb R$  for each $v\in V$, because $g(v,v) = {\overline {g(v,v)}}$.

An inner product space is a pair $(V,g)$.

If $(V, g)$ is an inner product vector space over $\mathbb C$, it is clear that $g$ is antilinear in its second argument, that is: $$g(v, \lambda v'_1 + \mu v'_2) = {\bar \lambda} g(v,v'_1) + {\bar \mu} g(v,v'_2)$$\noindent for all $v, v'_1,v'_2 \in V$, and $\bar \lambda$ and $\bar \mu$ being the conjugates of $\lambda$ and $\mu$ respectively.

Nevertheless, if $(V$, g) is an inner product vector space over $\mathbb R$, then $g$ is symmetric and bilinear.

The norm on an inner product vector space $(V,g)$ is the real-valued function $$\begin{aligned} \Vert \cdot \Vert_g \colon V &\longrightarrow \mathbb R \\ v &\longmapsto + \sqrt{g(v,v)} \, , \end{aligned}$$\noindent and the distance is the map $$\begin{aligned} d_g \colon V \times V &\longrightarrow \mathbb R \\ (v,v') &\longmapsto \Vert v' - v \Vert_g \, . \end{aligned}$$

Every inner product vector space $(V,g)$ has a natural structure of metric topological space determined by the distance $d_g$. Complete inner product $\mathbb C$-vector spaces are known as ``Hilbert spaces''. Usually, the inner product of a Hilbert space ${\mathcal H}$ is denoted by $<\cdot, \cdot>_{\mathcal H}$. Henceforth, we shall write $\mathcal H$ to refer to a Hilbert space and keep the inner product $<\cdot, \cdot>_{\mathcal H}$ implicit. 

Since a Banach space is a complete normed space, one has that each property of Banach spaces is valid for Hilbert spaces.

\subsubsection{Orthogonality}\quad

\begin{defn} \label{def:ortho-in-ve-3948} If  $(V,g)$ is an inner product vector space, we say that two vectors $v,v' \in V$ are orthogonal when $g(v,v') = 0 = g(v',v)$.
\end{defn}

\begin{defn} \label{def:perp-dub-basic98} Given a subspace $L$ of an inner vector space $(V,g)$, we shall call ``orthogonal of $L$'', $L^\perp$, to the subset of $V$ that consists of all vectors that are orthogonal to every $h\in L$, that is $$L^\perp = \{v\in V \text{ such that } g(v,h)= 0 \text{ for every } h\in L\}\, .$$
\end{defn} 

If $L \subseteq {\mathcal H}$ is a subspace of an arbitrary Hilbert space, it is known that $(S^\perp)^\perp = {\overline S}$ where ${\overline L}$ denotes the closure of $L$. Accordingly,  if $L \subseteq {\mathcal H}$ is closed, then $(L^\perp)^\perp = L$ and ${\mathcal H} = L \oplus L^\perp$.

A family $\{u_i\}_{i\in I}$ of orthonormal vectors of a Hilbert space $\mathcal H$ is called ``orthonormal basis'' when $\langle u_i \rangle_{i\in I}$ is dense in $\mathcal H$.

In general an orthonormal basis of $\mathcal H$ is not a Hamel basis of $\mathcal H$. Furthermore, it is known that every Hilbert space $\mathcal H$ admits orthonormal bases and all orthonormal bases of $\mathcal H$ have the same cardinality.  A Hilbert space $\mathcal H$ is named ``separable'' when it has a countable orthonormal basis.

\subsubsection{Bounded Operators}

We shall now recall the main properties of bounded operators of Hilbert spaces.

\begin{defn} \label{def:bound-ope-hilb09347} If ${\mathcal H}_1$ and ${\mathcal H}_2$ are two Hilbert spaces, a linear map $f\colon {\mathcal H}_1 \to {\mathcal H}_2$ is said ``bounded'' when there exists $C\in {\mathbb R}^+$ such that $$\Vert f(v) \Vert_{g_2} \leq C \cdot \Vert v \Vert_{g_1}\, ,$$\noindent for every $v\in {\mathcal H}_1$.
\end{defn}

We shall denote by $B({\mathcal H}_1, {\mathcal H}_2)$ the set of bounded linear maps $f\colon {\mathcal H}_1 \to {\mathcal H}_2$ and by $B(\mathcal H)$ the set of bounded endomorphisms of a Hilbert space $\mathcal H$. Given a linear map $f\in B({\mathcal H}_1, {\mathcal H}_2)$, it is known that $f$ is continuous if and only if $f$ is bounded. 

The sum and the composition of linear maps are operations on the set $B(\mathcal H)$. Also, the ``Bounded Inverse Theorem'' states that if $f \in B({\mathcal H}_1, {\mathcal H}_2)$ is bijective, then $f^{-1} \in B({\mathcal H}_2, {\mathcal H}_1)$.

Let us now consider two inner product vector spaces: $(V,g)$ and $(W,{\bar g})$. If $f\colon V \to W$ is a linear map, a linear operator $f^*\colon W \to V$ is called the adjoint of $f$ when $$g (f^*(w), v) = {\bar g} (w, f(v))\, ,$$\noindent for all $v\in V$ and $w\in W$. If $f\in \ed_k (V)$, we say that $f$ is self-adjoint when $f^* = f$.

The existence and uniqueness of the adjoint $f^*$ of a bounded (or equivalently a continuous) operator on arbitrary Hilbert spaces is immediately deduced from the Riesz Representation Theorem  and it is easy to check that $[\text{Im } f]^\perp = \Ker f^*$. Moreover, the adjoint of a bounded linear map is also bounded.

For the main properties of the adjoint operators on Hilbert spaces readers are referred to \cite[Chapter 10]{Ro}.

The spectrum of a bounded operator $f\in B(\mathcal H)$ consists of complex numbers $\lambda$ such that $f - \lambda \text{Id}$ is not invertible. We shall denote the spectrum of $f$ by $\sigma (f)$ and it is clear that every eigenvalue of $f$ is an element of $\sigma (f)$. It is known that it is possible that an element of $\sigma (f)$ is not an eigenvalue.

\begin{defn} \label{def:compa-ope-roe36} Given a Hilbert space $\mathcal H$, a bounded operator $f\in B(\mathcal H)$ is compact if for every bounded sequence $\{h_n\}_{n\in \mathbb N} \subset {\mathcal H}$, the sequence $\{f(h_n)\}_{n\in \mathbb N}$ has a convergent subsequence. We say that $f$ is quasi-compact if $f^n$ is compact for same $n\in \mathbb N$.
\end{defn}

If $\mathcal H$ is an infinite-dimensional Hilbert space, and operator $f\in \ed_{\mathbb C} (\mathcal H)$ is compact if, for every $h\in \mathcal H$, it can be written in the form $$f (h) = \sum_{n\in \mathbb N} \gamma_n <v_n, h>_{\mathcal H} \cdot u_n$$\noindent where $\{u_n\}_{n\in \mathbb N}$ and $\{v_n\}_{n\in \mathbb N}$ are orthonormal bases of $\mathcal H$ and $\{\gamma_n\}_{n\in \mathbb N}$ is a convergent sequence of positive numbers with limit zero.

We shall denote by $C (\mathcal H)$ the set of compact operators on a Hilbert space that is a two-sided ideal of $B(\mathcal H)$.

\begin{defn} \label{de:trace-class-394346} A compact operator $f \in C(\mathcal H)$ is of trace class when $$\sum_{i\geq 1} \lambda_i (f) < \infty$$\noindent where $\{\lambda_i (f)\}$ is the listing of all non-zero eigenvalues of $f$, counted up to algebraic
multiplicity.
\end{defn}

The space of trace class operators on an arbitrary Hilbert space $\mathcal H$ is also an ideal of $B(\mathcal H)$ and it will be denoted by $T(\mathcal H)$. Every bounded finite rank endomorphism of a Hilbert space is of trace class.

If $\mathcal H$ is an arbitrary Hilbert space  and $\{u_i\}_{i\in I}$ is an orthonormal basis of $\mathcal H$, the trace of a trace class operator $f\in B(\mathcal H)$ is defined by the expression $$\tr (f) = \sum_{i\in I} <f(u_i), u_i>_{\mathcal H}\, .$$

It is known that $\tr(f)$ is independent of the choice of the orthonormal basis made, and, with the notation of Definition \ref{de:trace-class-394346}, V. B. Lidskii shows in \cite{Lid} that \begin{equation} \label{eq:lid-tra-3937} \tr (f) = \sum_{i\geq 1} \lambda_i (f) \in {\mathbb C}\, .\end{equation}

\smallskip
\subsubsection{The Leray Trace} \label{ss:Leray-tra-39373} Let $V$ be an arbitrary $k$-space and $f\in \ed_k (V)$.  If we write $$N(f) =  \bigcup_{s=1}^\infty \Ker f^s\, ,$$\noindent since $N(f)$ is an $f$-invariant subspace of $V$, we can consider the endomorphism ${\tilde f} \in \ed_k (V/N(f))$ induced by $f$.

Now, when $E_f = V/N(f)$ is a finite-dimensional $k$-vector space, according to the statements of \cite[Section 1]{Ler} the ``Leray trace'' $\tr^L_V$ is defined by $$\tr^L_V (f) = \tr_{E_f} (\tilde f)\, ,$$\noindent where $\tr_{E_f}$ is the usual trace of an endomorphism on $E_f$.

If $V'\subset V$ is a $f$-invariant subspace, $f' = f_{\vert_{V'}}$ and $f''$ is the induced linear map on $V/V'$, the Leray trace satisfies that $$\tr^L_V (f) = \tr^L_{V'} (f') + \tr^L_{V/V'} (f'')\, .$$

\smallskip

\subsubsection{Riesz Operators} \quad

Let $E$ be a complex Banach space and let $T$ be a bounded operator on $E$.

\begin{defn} \cite[Definition 3.1]{El} \label{d:Riesz-point-34937} We say that $\lambda \in \sigma (T)$ is a ``Riesz point'' for $T$ if $E$ is a direct sum $$E= N(\lambda) \oplus F(\lambda)$$\noindent where:

\begin{enumerate}

\item $E(\lambda)$ and $F(\lambda)$ are $T$-invariant linear subspaces of $E$;

\item $N(\lambda)$ is finite dimensional;

\item $F(\lambda)$ is closed;

\item $T - \lambda \text{Id}$ is nilpotent on $N(\lambda)$;

\item $T - \lambda \text{Id}$ is a homeomorphism of $F(\lambda)$.
\end{enumerate}
\end{defn}

\begin{defn} \cite[Definition 3.2]{El} \label{d:Riesz-operator-334344937} A bounded operator $T$ on a complex Banach space is a ``Riesz operator'' if every non-zero point of its spectrum is a Riesz point.
\end{defn}

According to \cite[Theorem 2.1]{We1}, it is known that if $\lambda$ is a non-zero Riesz point of $\sigma (T)$ then $\lambda$ is isolated in $\sigma (T)$.

\begin{defn} \label{df:quasi-3847-nilpo} We say that a bounded operator $T$ on a complex Banach space is ``quasinilpotent'' when $\sigma (T) = \{0\}$.
\end{defn}

The fully decomposition of a Riesz operator introduced by T. T. West in \cite{We2} is

\begin{defn} \label{def-full-decomopso3736} If $T$ is a Riesz operator on a Banach space $E$, $T$ is said to be ``fully decomposable'' if $T = T_{_C} + T_{_Q}$, where $T_{_C}$ is a compact operator, $T_{_Q}$ is quasi-nilpotent and $T_{_C} \circ T_{_Q} = T_{_Q} \circ T_{_C} = 0$.
\end{defn}

A decomposition $T = T_{_C} + T_{_Q}$ is known as ``West decomposition of $T$''.

Moreover, it follows from \cite[Theorem 3.8]{El} that for every Riesz operator $T$ on a Hilbert space $\mathcal H$ then $T = T_{_C} + T_{_Q}$, where $T_{_C}$ is a compact operator, $T_{_Q}$ is quasi-nilpotent, $T_{_C}$ is normal, that is $\sigma (T) = \sigma (T_{_C})$,  and the non-zero eigenvalues of $T$ and $T_{_C}$ have the same algebraic multiplicities.

\begin{defn} \cite[Definition 4.6]{El} \label{def:Riesz-trace-class-trace-292737} Suppose that $T$ is a Riesz operator on a Hilbert space $\mathcal H$ and $T = T_{_C} + T_{_Q}$ is a West decomposition of $T$. If $T_{_C}$ is of trace class then we say that $T$ is of Riesz trace class and we define $$\tr^R_{\mathcal H} (T) = \tr (T_{_C})\, ,$$\noindent where $\tr (T_{_C})$ is the trace of the trace class operator $T_{_C}$.
\end{defn} 
 Moreover, if $T$ is Riesz trace class operator on a Hilbert space $\mathcal H$, according to the statements of \cite[Section 4]{El}, the trace $\tr^R_{\mathcal H} (T)$ satisfies the following properties:

\begin{enumerate}

\item the listing $\{\lambda_i (T)\}_{i\in I}$  of the non-zero eigenvalues of $T$, repeated according to multiplicity, is finite and $\tr^R_{\mathcal H} (T) = \sum \lambda_i (T)$;

\item if $f$ is a bounded map on $\mathcal H$ with a bounded inverse, then $$\tr^R_{\mathcal H} (T) = \tr^R_{\mathcal H} (f\circ T \circ f^{-1})\, ;$$

\item if $T^*$ is the adjoint of $T$, then $T^*$ is Riesz trace class and $\tr^R_{\mathcal H} (T) = {\overline {\tr^R_{\mathcal H} (T^*)}}$;

\item if $g$ is a bounded operator on $\mathcal H$ such that $g\circ T = T\circ g$, then $g\circ T$ and $T\circ g$ are of Riesz trace class en $\tr^R_{\mathcal H} (g \circ T) = \tr^R_{\mathcal H} (T\circ g)$.

\end{enumerate}

\medskip 

\subsection{Finite Potent Endomorphisms} \label{ss:finite-potent}

Let $k$ be an arbitrary field, and let $V$ be a $k$-vector space.

    Let us now consider an endomorphism $\varphi$ of $V$. According to \cite[page 149]{Ta}, we say
that $\varphi$ is ``finite potent'' if $\varphi^n V$ is finite
dimensional for some $n$.

 In 2007 M. Argerami, F. Szechtman and R. Tifenbach showed in \cite{AST} that an endomorphism $\varphi$ is
finite potent if and only if $V$ admits a $\varphi$-invariant decomposition $V = U_\varphi \oplus W_\varphi$ such that
$\varphi_{\vert_{U_\varphi}}$ is nilpotent, $W_\varphi$ is finite
dimensional, and $\varphi_{\vert_{W_\varphi}} \colon W_\varphi
\overset \sim \longrightarrow W_\varphi$ is an isomorphism.

Indeed, if $k[x]$ is the algebra of polynomials in the variable $x$ with coefficients in $k$, we may view $V$ as an $k[x]$-module via $\varphi$, and the explicit definition of the above $\varphi$-invariant subspaces of $V$ is:
\begin{itemize}

\item $U_\varphi = \{v \in V \text{ such that } \varphi^m (v) = 0 \text{ for some m }\}$;

\item $W_\varphi = \left  \{\begin{aligned} &\qquad v \in V \text{ such that } p(\varphi) (v) = 0 \text{ for } \\ &\text{some } p(x) \in k[x] \text{ relatively prime to } x\end{aligned} \right \}$.

\end{itemize}

Note that if the annihilator polynomial of $\varphi$ is $x^m\cdot p(x)$ with $(x,p(x)) = 1$, then $U_\varphi = \Ker \varphi^m$ and $W_\varphi = \Ker p(\varphi)$.

Hence, this decomposition is unique. We shall call this decomposition the $\varphi$-invariant AST-decomposition of $V$.

Moreover,  we shall call ``index of $\varphi$'', $i(\varphi)$, to the nilpotent order of $\varphi_{\vert_{U_\varphi}}$, which coincides with the smaller $n\in \mathbb N$ such that  $\text{Im } \varphi^n = W_\varphi$. One has that $i(\varphi) = 0$ if and only if $V$ is a finite-dimensional vector space and $\varphi$ is an automorphism.

\begin{lem} \label{lem:inv-sub-AST-349346} If $V$ is $k$-vector space, $\varphi \in \ed_k (V)$ is a finite potent endomorphism with AST-decomposition $V = W_{_\varphi} \oplus U_{_\varphi}$ and $L\subset V$ is $\varphi$-invariant, then one has that:
\begin{itemize}

\item if $\varphi_{\vert_L} \in \aut_k (L)$, then $L$ is finite-dimensional and $L\subseteq W_{_\varphi}$;

\item if $\varphi_{\vert_L}$ is nilpotent, then $L\subseteq U_{_\varphi}$.

\end{itemize}
\end{lem}

\begin{proof} The statements are direct consequence of the uniqueness of the AST-decompo\-sition of $\varphi$.
\end{proof}

Basic examples of finite potent endomorphisms are all endomorphisms of a finite-dimensional vector spaces and finite rank or nilpotent endomorphisms of infinite-dimensional vector spaces.

\begin{defn} \label{def:tat-tra-39336} For a finite potent endomorphism $\varphi \in \ed_k (V)$, a trace $\tr_V(\varphi) \in k$ may
be defined from the following properties:
\begin{enumerate}
\item if $V$ is finite dimensional, then $\tr_V(\varphi)$ is the ordinary trace;
 \item if $W$ is a subspace of $V$ such that $\varphi W \subset W$, then $$\tr_V(\varphi) = \tr_W(\varphi) +
\tr_{V/W}(\varphi)\, ;$$ \item if $\varphi$ is nilpotent, then
$\tr_V(\varphi) = 0$.
\end{enumerate}

\end{defn}

Usually, $\tr_V$ is named ``Tate's trace''. 

It is known that in general $\tr_V$ is not linear; that is, it is possible to find finite potent endomorphisms $\theta_1, \theta_2 \in \ed_k (V)$ such that $$\tr_V (\theta_1 + \theta_2) \ne \tr_V (\theta_1) + \tr_V (\theta_2)\, .$$

For details readers are referred to \cite{Pa1}, \cite{RPa} and \cite{Ta}.
\medskip

\subsection{Core-Nilpotent Decomposition of a Finite Potent Endomorphism}

Let $V$ be again an arbitrary $k$-vector space. Given a finite potent endomorphism $\varphi \in \ed_k (V)$,  there exists a unique decomposition $\varphi = \varphi_{_1} + \varphi_{_2}$, where $\varphi_{_1}, \varphi_{_2} \in \ed_k (V)$ are finite potent endomorphisms satisfying that:

\begin{itemize}

\item $i(\varphi_{_1}) \leq 1$;

\item $\varphi_{_2}$ is nilpotent;

\item $\varphi_{_1} \circ \varphi_{_2} = \varphi_{_2} \circ \varphi_{_1} = 0$.

\end{itemize}

According to \cite[Theorem 3.2]{Pa-CN}, if $\varphi^D$ is the Drazin inverse of $\varphi$ offered in \cite{Pa-Dr}, one has that $\varphi_1 = \varphi \circ \varphi^D \circ \varphi$ is the core part of $\varphi$. Also, $\varphi_2$ is named the nilpotent part of $\varphi$ and one has that \begin{equation} \label{eq:index1} \varphi = \varphi_1 \Longleftrightarrow U_\varphi = \Ker \varphi \Longleftrightarrow  W_\varphi = \text{ Im } \varphi \Longleftrightarrow (\varphi^D)^D = \varphi \Longleftrightarrow i(\varphi) \leq 1\, .\end{equation}

Moreover, if $V = W_{_\varphi}\oplus U_{_\varphi}$ is the AST-decomposition of $V$ induced by $\varphi$, then $\varphi_{_1}$ and $\varphi_{_2}$ are the unique linear maps such that:

\begin{equation} \label{eq:expl-CN-exp-3498353} \varphi_{_1} (v) = \left \{ \begin{aligned} \varphi (v) \, &\text{ if } \, v\in W_{_\varphi} \\ \, 0 \quad &\text{ if } \, v\in U_{_\varphi} \end{aligned} \right . \quad \text{ and } \quad \varphi_{_2} (v) = \left \{ \begin{aligned} \, 0 \quad &\text{ if } \, v\in W_{_\varphi} \\ \varphi (v) \, &\text{ if } \, v\in U_{_\varphi} \end{aligned} \right . \quad \, .\end{equation}

\medskip

\section{Bounded finite potent endomorphisms} \label{s:345-boun-fp}

In this section we shall study the main properties of bounded finite potent endomorphisms on an arbitrary Hilbert space ${\mathcal H}$.

Let us consider a finite potent endomorphism $\varphi \in \ed_{\mathbb C} ({\mathcal H})$ with CN-decomposi\-tion $\varphi = \varphi_{_1} + \varphi_{_{_2}}$.

In general, a finite potent endomorphism is not bounded. In fact, there exist finite rank endomorphisms and nilpotent endomorphisms that are not bounded, as it is deduced from the following counter-example.

Let $\mathcal H$ be a separable Hilbert space  with orthonormal basis $\{u_i\}_{i\in \mathbb N}$ and let us consider the linear map $f\in \ed_k (\mathcal H)$ defined from the assignations $$f(u_i) = \left \{ \begin{aligned} 0\quad  \, \quad \text{ if } i = 1 \\ i\cdot  u_1 \, \quad \text{ if } i \geq 2   \end{aligned} \right . \, .$$ One has that $f$ is nilpotent of finite rank and it is not bounded.

Henceforth, we shall write $B_{fp} (\mathcal H)$ to refer to the set of bounded finite potent endomorphisms of an arbitrary Hilbert space $\mathcal H$.

\begin{rem} \label{re:Bfpisnot-29337} Let $\mathcal H$ be a separable Hilbert space  with orthonormal basis $\{u_i\}_{i\in \mathbb N}$ and let us consider $\varphi, \varphi' \in B_{fp} (\mathcal H)$, defined from the assignations: 
$$\varphi (u_i) = \left \{ \begin{array}{ccl} \frac1{i^2} u_{i+1} & \text{ if i is odd} \\  0 & \text{ if i is even}   \end{array} \right .$$\noindent and $$\varphi' (u_i) = \left \{ \begin{array}{ccl} 0 & \text{ if i is odd} \\  \frac{1}{i^2} u_{i-1} & \text{ if i is even}   \end{array} \right . \, .$$

Bearing in mind that $$(\varphi + \varphi') (u_i) = \left \{ \begin{array}{ccl} \frac1{i^2} u_{i+1} & \text{ if i is odd} \\   \frac{1}{i^2} u_{i-1} & \text{ if i is even}   \end{array} \right .$$\noindent and $$(\varphi' \circ \varphi) (u_i) = \left \{ \begin{array}{ccl} \frac1{i^2 (i+1)^2} u_{i} & \text{ if i is odd} \\   0 & \text{ if i is even}   \end{array} \right . \, ,$$\noindent it is clear that $\varphi + \varphi', \varphi' \circ \varphi \notin B_{fp} (\mathcal H)$ and, therefore, $B_{fp} (\mathcal H)$ is not an ideal of $B (\mathcal H)$.
\end{rem}

\begin{lem} \label{l:goun-esou.duy63} Given a Hilbert space $\mathcal H$ with a decomposition $\mathcal H = M \oplus N$, where $M$ and $N$ are closed subspaces,  and given an endomorphism $f\in \ed_{\mathbb C} (\mathcal H)$ such that $f_{\vert_{M}} \in B(M,{\mathcal H})$, then the linear operator $f_{_M} \in \ed_{\mathbb C} (\mathcal H)$ defined as $$f_{_M} (v) = \left  \{ \begin{aligned} f (v) \, &\text{ if } \, v\in M \\ \, 0 \quad &\text{ if } \, v\in N \end{aligned} \right .$$\noindent is bounded.
\end{lem}

\begin{proof} If we denote by ${\mathcal P}_{M,N} \in \ed_{\mathbb C} (\mathcal H)$ the oblique projection of $\mathcal H$ onto $M$ along $N$, since $f_{_M} = f_{\vert_{M}} \circ {\mathcal P}_{M,N}$ and ${\mathcal P}_{M,N} \in B (\mathcal H)$ because $M$ and $N$ are closed, then we conclude that $f_{_M}  \in B(\mathcal H)$.
\end{proof}

If $V$ is an arbitrary Banach space, bearing in mind that the oblique projection ${\mathcal P}_{M,N}\in B(\mathcal H)$ when $V = M\oplus N$ with closed subspaces $M$ and $N$, one has that Lemma \ref{l:goun-esou.duy63} hold for Banach spaces.

\begin{lem} \label{ex-boun-nilp-U-38463} If $\mathcal H$ is a Hilbert space, $f \in \ed_{\mathbb C} ({\mathcal H})$ and $U \subseteq \mathcal H$ is a closed subspace of finite codimension such that $f_{\vert_U} = 0$, then $f\in B(\mathcal H)$.
\end{lem}

\begin{proof} Since $U$ is closed, then ${\mathcal H} = U \oplus U^\perp$. Moreover, since $U$ is of finite codimension, one has that $U^\perp$ is finite-dimensional and there exists $C_{_{U^\perp}} \in {\mathbb R}^+$ such that $$\Vert f(v) \Vert_{{\mathcal H}} \leq C_{_{U^\perp}} \cdot \Vert v \Vert_{{\mathcal H}}$$\noindent for every $v\in U^\perp$.

Hence, given now $h\in \mathcal H$, such that $h = v + u$ with $v\in U^\perp$ and $u\in U$, and bearing in mind that $\Vert v \Vert_{{\mathcal H}} \leq \Vert h \Vert_{{\mathcal H}}$, one has that $$\Vert f(h) \Vert_{{\mathcal H}} \leq \Vert f(v) \Vert_{{\mathcal H}}  \leq C_{_{U^\perp}} \cdot \Vert v \Vert_{{\mathcal H}}  \leq C_{_{U^\perp}} \cdot \Vert h \Vert_{{\mathcal H}}\, ,$$\noindent from where we deduce that $f$ is bounded.
\end{proof}

\begin{lem} \label{l:core-part-is-bunded-394363}  If $\mathcal H$ is a Hilbert space and we consider a finite potent endomorphism $\varphi \in\ed_{\mathbb C} ({\mathcal H})$ with CN-decomposition $\varphi = \varphi_{_1} + \varphi_{_2}$, then $\varphi_{_1} \in B_{fp} (\mathcal H)$.
\end{lem}

\begin{proof} Let ${\mathcal H} = W_{_\varphi} \oplus U_{_\varphi}$ be the AST-decomposition induced by $\varphi$. If $i(\varphi) = n$, since $U_{_\varphi} = \Ker \varphi^n$ and $\varphi^n$ is bounded, we deduce that $U_{_\varphi}$ is a closed subspace of $\mathcal H$ of finite codimension. 

Thus, bearing in mind the explicit expression of the finite potent endomorphism $\varphi_{_1}$ offered in (\ref{eq:expl-CN-exp-3498353}),  the statement is immediately deduced from Lemma \ref{ex-boun-nilp-U-38463}.
\end{proof}

\begin{cor} \label{c:vparhi-fn-rank-39437} Given a Hilbert space $\mathcal H$ a  a finite potent endomorphism $\varphi \in\ed_{\mathbb C} ({\mathcal H})$ with CN-decomposition $\varphi = \varphi_{_1} + \varphi_{_2}$, then $\varphi_{_1}$ is a bounded finite rank operator on $\mathcal H$.
\end{cor}

\begin{proof} Bearing in mind that $\varphi_1$ is of finite rank because $i(\varphi_1) \leq 1$, the assertion is immediately deduced from Lemma \ref{l:core-part-is-bunded-394363}.
\end{proof}

\begin{lem} \label{l:eq-cor-nilprt-is-bunded-3963}  If $\mathcal H$ is a Hilbert space and we consider a finite potent endomorphism $\varphi \in\ed_{\mathbb C} ({\mathcal H})$ with CN-decomposition $\varphi = \varphi_{_1} + \varphi_{_2}$, then $\varphi \in B_{fp}(\mathcal H)$ if and only if $\varphi_{_2} \in B_{fp} (\mathcal H)$.
\end{lem}

\begin{proof} Bearing in mind that from Lemma \ref{l:core-part-is-bunded-394363} we know that $\varphi_{_1} \in B_{fp}$, if  $\varphi_{_2} \in B_{fp} (\mathcal H)$, since $\varphi = \varphi_{_1} + \varphi_{_2}$, then $\varphi \in B_{fp} (\mathcal H)$.

Conversely, if $\varphi \in B_{fp} (\mathcal H)$, from Lemma \ref{l:core-part-is-bunded-394363}, we know that  $\varphi_{_1} \in B_{fp}(\mathcal H)$. Accordingly, we have that the finite potent endomorphism $\varphi_{_2} \in B_{fp} (\mathcal H)$ because $\varphi_{_2} = \varphi - \varphi_{_1}$ and the claim is proved.
\end{proof}

\begin{thm}[Characterization of bounded finite potent endomorphisms] \label{th:char-boun-fp30363} Given a Hilbert space $\mathcal H$ and an endomorphism $\varphi \in\ed_{\mathbb C} ({\mathcal H})$, then the following conditions are equivalent:

\begin{enumerate}

\item $\varphi \in B_{fp} (\mathcal H)$;

\item $\mathcal H$ admits a decomposition $\mathcal H = W_{_\varphi} \oplus U_{_\varphi}$ where $W_{_\varphi}$ and $U_{_\varphi}$ are closed $\varphi$-invariant subspaces of $\mathcal H$, $W_{_\varphi}$ is finite-dimensional, $\varphi_{\vert_{W_{_\varphi}}}$ is an homeomorphism of $W_{_\varphi}$ and $\varphi_{\vert_{U_{_\varphi}}}$ is a bounded nilpotent operator.

\item $\varphi$ has a decomposition $\varphi = \psi + \phi$, where $\psi$ is a bounded finite rank operator, $\phi$ is a bounded nilpotent operator and $\psi \circ \phi = \phi \circ \psi = 0$.
\end{enumerate}
\end{thm}

\begin{proof} $1) \Longrightarrow 2)$ If $\varphi \in B_{fp} (\mathcal H)$, $a_\varphi (x) = x^n p(x)$ is the annihilator  polynomial of $\varphi$ and we consider the AST-decomposition of $\mathcal H= W_{_\varphi} \oplus U_{_\varphi}$ determined by $\varphi$, we have that:

\begin{itemize}

\item $W_{_\varphi} = \Ker p(\varphi)$ is finite-dimensional, $\varphi$-invariant and closed;

\item $U_{_\varphi} = \Ker \varphi^n$ is $\varphi$-invariant and closed;

\item $\varphi_{\vert_{W_{_\varphi}}}$ and $\varphi_{\vert_{U_{_\varphi}}}$ are bounded because the restriction of a bounded operator to a closed subspace is also bounded;

\item from the Bounded Inverse Theorem, since $\varphi_{\vert_{W_{_\varphi}}} \in \aut_{\mathbb C} (W_{_\varphi})$, then $\varphi_{\vert_{W_{_\varphi}}}$ is an homeomorphism of $W_{_\varphi}$.

\end{itemize}

$2) \Longrightarrow 3)$ If $\mathcal H$ admits a decomposition $\mathcal H = W_{_\varphi} \oplus U_{_\varphi}$ satisfying the conditions of the second paragraph of this theorem, if we denote $\psi = \varphi_{_{W_\varphi}}$ and $\phi = \varphi_{_{U_\varphi}}$ with $$\varphi_{_{W_\varphi}}  (v) = \left  \{ \begin{aligned} \varphi (v) \, &\text{ if } \, v\in W_\varphi \\ \, 0 \quad &\text{ if } \, v\in U_\varphi \end{aligned} \right . \quad \text{ and } \quad  \varphi_{_{U_\varphi}}  (v) = \left  \{ \begin{aligned} \varphi (v) \, &\text{ if } \, v\in U_\varphi \\ \, 0 \quad &\text{ if } \, v\in W_\varphi \end{aligned} \right .\, ,$$\noindent from Lemma \ref{l:goun-esou.duy63} we have that $\psi$ and $\phi$ are bounded and, clearly, $\psi$ is of finite rack and $\phi$ is nilpotent.

$3) \Longrightarrow 1)$ Let us now assume that $\varphi$ has a decomposition $\varphi = \psi + \phi$, where $\psi$ is a bounded finite rank operator, $\phi$ is a bounded nilpotent operator and $$\psi \circ \phi = \phi \circ \psi = 0\, .$$

From this decomposition, one immediately has that $\varphi \in B(\mathcal H)$ and, since $\varphi^n = \psi^n$ for $n>>0$, we deduce that $\varphi \in B_{fp} (\mathcal H)$.
\end{proof}

From the uniqueness of the CN-decomposition $\varphi = \varphi_{_1} + \varphi_{_2}$ proved in \cite[Theorem 3.2]{Pa-CN}, if $\varphi = \psi + \phi$ as in Theorem \ref{th:char-boun-fp30363}, one has that $\psi = \varphi_{_1}$ and $\phi = \varphi_{_2}$.

Recall now that the ``Invariant Subspace Problem'' is referred to give an answer to the following question:  is there a $T$-invariant non-trivial closed subspace of $E$, if $T$ is a bounded operator on a complex Banach space $E$?

\begin{prop} \label{prop:inv-sub-probl-hil-3094734} If $\mathcal H$ is an infinite-dimensional Hilbert space and $\varphi \in B_{fp} (\mathcal H)$ with $i(\varphi) \geq 2$, then we have an affirmative answer to the Invariant Subspace Problem for $\varphi$ . Moreover, if $\hat \varphi \in B_{fp} (\mathcal H)$ with $i(\hat \varphi) = 1$, then $\hat \varphi$ gives an affirmative answer to the Invariant Subspace Problem if and only if $\hat \varphi$ is not nilpotent.
\end{prop}

\begin{proof} Let ${\mathcal H} = W_{_\varphi} \oplus U_{_\varphi}$ is the AST-decomposition induced by $\varphi \in B_{fp} (\mathcal H)$ with $i(\varphi) \geq 1$. If $\varphi$ is not nilpotent, one has that $W_{_\varphi}$ is a $\varphi$-invariant non-trivial closed subspace of $\mathcal H$.

Let us now consider a bounded nilpotent endomorphism ${\tilde \varphi}$ with $i({\tilde \varphi}) = r \geq 2$. In this case, we have that $\Ker {\tilde \varphi}^{r-1}$ is a ${\tilde \varphi}$-invariant non-trivial closed subspace of $\mathcal H$.

Finally, if ${\hat \varphi} \in B_{fp} (\mathcal H)$ is a nilpotent endomorphism with $i({\hat \varphi}) = 1$, one has that ${\mathcal H} = \Ker {\hat \varphi}$ and it is clear that the unique ${\hat \varphi}$-invariant subspace is $\{0\}$.
\end{proof}

Our task is now to study compact finite potent endomorphisms on arbitrary Hilbert spaces.

Firstly, it is known that bounded finite rank endomorphisms of Hilbert spaces are compact but, in general, a bounded nilpotent endomorphism of a Hilbert space is not compact. An easy counter-example is the following: if $\mathcal H$ is a separable Hilbert space and $\{u_i\}_{i\in \mathbb N}$ is an orthonormal basis of $\mathcal H$, then the linear operator $f\in B_{fp} (\mathcal H)$ determined by the conditions $$f(u_i) =  \left \{\begin{array}{ccl} u_{i+1}  & \text{ if } & \text{i is odd} \\ 0 & \text{ if }  & \text{i is even} \end{array}  \right .$$\noindent is nilpotent and it is clear that it is not compact.

\begin{prop} \label{pro:chr-compa.bounded03436} If $\mathcal H$ is a Hilbert space and  we consider $\varphi \in B_{fp} (\mathcal H)$ with CN-decomposition $\varphi = \varphi_{_1} + \varphi_{_2}$, then $\varphi \in C(\mathcal H)$ if and only if $\varphi_{_2}\in C(\mathcal H)$.
\end{prop}

\begin{proof} Since we know from Corollary \ref{c:vparhi-fn-rank-39437} that $\varphi_{_1}$ is a bounded finite rank operator of $\mathcal H$, then $\varphi_1 \in C(\mathcal H)$ and we conclude bearing in mind that $C(\mathcal H)$  is an ideal of $B(\mathcal H)$.  
\end{proof}

\begin{cor} \label{c:inx-1-is-compac-3494363} Given a Hilbert space $\mathcal H$ and $\varphi$ is a bounded finite rank operator of $\mathcal H$ with $i(\varphi) \leq 1$, then $\varphi \in C(\mathcal H)$.
\end{cor}

\begin{proof} This statement is a direct consequence of Proposition \ref{pro:chr-compa.bounded03436} because $i(\varphi) \leq 1$ if and only if $\varphi = \varphi_1$.
\end{proof}

\begin{lem} \label{l:bounde-quaasi-compa-34937} Every bounded finite potent endomorphism on a Hilbert space is quasi-compact.
\end{lem}

\begin{proof} Bearing in mind Definition \ref{def:compa-ope-roe36}, since $\varphi^n = (\varphi_{_1})^n$ for every $n\geq i(\varphi)$, the claim follows from Corollary \ref{c:inx-1-is-compac-3494363}.
\end{proof}

We shall now study the spectrum of a finite potent bounded endomorphism.

\begin{lem} \label{l:eigen-fin-po-core-38336} Given a Hilbert space $\mathcal H$ and an endomorphism $\varphi \in B_{fp} (\mathcal H)$ with AST-decomposition ${\mathcal H} = W_{_\varphi} \oplus U_{_\varphi}$ induced by $\varphi$, then a non-zero $\lambda \in \mathbb C$ is an eigenvalue of $\varphi$ if and only if $\lambda$ is an eigenvalue of $\varphi_{\vert_{W_{_\varphi}}}$.
\end{lem}

\begin{proof} It is clear that if $\lambda$ is an eigenvalue of $\varphi_{\vert_{W_{_\varphi}}}$ then $\lambda$ is also an eigenvalue of $\varphi$.

Conversely, let us assume that $\lambda$ is an eigenvalue of $\varphi$ and let us consider a non-zero vector $v\in V$ such that $\varphi (v) = \lambda \cdot v$. 

Thus, since $\langle v \rangle$ satisfies that $\varphi_{\vert_{\langle v \rangle}} \in \aut_{\mathbb C} (\langle v\rangle)$, from Lemma \ref{lem:inv-sub-AST-349346} one deduces that $v\in W_{_\varphi}$ and, therefore, $\lambda$ is an eigenvalue of $\varphi_{\vert_{W_{_\varphi}}}$.
\end{proof}

\begin{lem} \label{l:eigenv-core-part-393373}  If $\mathcal H$ is a Hilbert space and we consider $\varphi \in B_{fp} (\mathcal H)$ with CN-decomposition $\varphi = \varphi_{_1} + \varphi_{_2}$, then $\lambda \in \mathbb C$ is an eigenvalue of $\varphi$ if and only if $\lambda$ is an eigenvalue of $\varphi_1$.
\end{lem}

\begin{proof} Bearing in mind the explicit expression of $\varphi_1$ offered in (\ref{eq:expl-CN-exp-3498353}), the claim is immediately deduced from Lemma \ref{l:eigen-fin-po-core-38336}.
\end{proof}

\begin{prop} \label{pro:spectru-finit-p3i3353} If $\mathcal H$ is a Hilbert space, $\varphi \in B_{fp} (\varphi)$ and $\mathcal H = W_{_\varphi} \oplus U_{_\varphi}$ is the AST-decomposition determined by $\varphi$, one has that the spectrum of $\varphi$ is:

\begin{itemize}

\item $\sigma (\varphi)  = \{\lambda_1, \dots, \lambda_n\}$ when $i(\varphi) = 0$;

\item $\sigma (\varphi)  = \{0,\lambda_1, \dots, \lambda_n\}$ when $i(\varphi) \geq 1$,
\end{itemize}

where $\{\lambda_1, \dots, \lambda_n\}$ are the eigenvalues of $\varphi_{\vert_{W_{_\varphi}}}$.
\end{prop}

\begin{proof} Recalling that $i(\varphi) = 0$ if and only if $\mathcal H$ is finite-dimensional and $\varphi \in \aut_{\mathbb C} (\mathcal H)$, it is clear that the spectrum of $\varphi$ coincides with the set of eigenvalues of $\varphi$ because, in this case, $\mathcal H = W_{_\varphi}$.

Let us assume that $i(\varphi) \geq 1$. Then, if we consider a non-zero $\lambda \in {\mathbb C}$ such that $\lambda$ is not an eigenvalue of $\varphi_{\vert_{W_{_\varphi}}}$, since $W_{_\varphi}$ and $U_{_\varphi}$ are invariants under the action of $\varphi$, one has that $\lambda \cdot \text{Id} - \varphi_{\vert_{W_{_\varphi}}} \in \aut_{\mathbb C} (W_{_\varphi})$ and $\lambda \cdot \text{Id} - \varphi_{\vert_{U_{_\varphi}}} \in \aut_{\mathbb C} (U_{_\varphi})$, from where we deduce that $\lambda \cdot \text{Id} - \varphi$ is invertible. 

Hence, bearing in mind that $\varphi$ is not invertible when $i(\varphi) \geq 1$ and the same holds for $\lambda_i \cdot \text{Id} - \varphi$ for each eigenvalue $\lambda_i$ of $\varphi_{\vert_{W_{_\varphi}}}$, the statement is proved.
\end{proof}

A direct consequence of this proposition is:

\begin{cor} \label{c:spectru,-varp-core-203453} If $\mathcal H$ is a Hilbert space and  we consider $\varphi \in B_{fp} (\mathcal H)$ with CN-decomposition $\varphi = \varphi_{_1} + \varphi_{_2}$, then $\sigma (\varphi) = \sigma (\varphi_1)$.
\end{cor}

Moreover, one has that:

\begin{lem} \label{l:finite-spectu35343j} If $\mathcal H$ is a Hilbert space and  we consider $\varphi \in B_{fp} (\mathcal H)$, then the spectrum satisfies the following properties:
\begin{enumerate}

\item $\sigma (\varphi)$ is finite;

\item $\lambda \in \sigma (\varphi)$ if and only if $\lambda$ is an eigenvalue of $\varphi$;

\item $\text{dim}_{\mathbb C} \, \Ker (\varphi - \lambda \text{Id}) < \infty$ for every $0\ne \lambda \in \sigma (\varphi)$.

\end{enumerate}
\end{lem}

\begin{proof} The assertions follows from Lemma \ref{l:eigen-fin-po-core-38336} and Proposition \ref{pro:spectru-finit-p3i3353}.
\end{proof}

\begin{lem} \label{l:oinci-eig-alg-38363} If $\mathcal H$ is a Hilbert space, $\varphi \in B_{fp} (\mathcal H)$ and $0\ne \lambda\in \sigma (\varphi)$, then the algebraic multiplicity of $\lambda$ as an eigenvalue of $\varphi_1$ coincides with the algebraic multiplicity of $\lambda$ as an eigenvalue of $\varphi_{\vert_{W_{_\varphi}}}$.
\end{lem}

\begin{proof} For every non-zero $\lambda \in {\mathbb C}$ and for each $n\in \mathbb N$, since $\Ker \, (\varphi - \lambda \text{Id})^n \subseteq W_{_\varphi}$, one has that $$\Ker \, (\varphi - \lambda \text{Id})^n = \Ker \, (\varphi_{\vert_{W_{_\varphi}}} - \lambda \text{Id})^n\, ,$$\noindent from where the claim is proved.
\end{proof}

If $\varphi \in B_{fp} (\mathcal H)$ and $\{\lambda_i (\varphi)\}$ is the listing of all non-zero eigenvalues of $\varphi$, counted up to algebraic multiplicity, then $\# \{\lambda_i (\varphi)\} = \text{dim}_{\mathbb C} \, W_{_\varphi}$.

\begin{prop} \label{prop:ev-boun-dp-isrisk-394373} Every bounded finite potent endomorphism on a Hilbert space is a Riesz operator.
\end{prop}

\begin{proof} Let $\mathcal H$ be a Hilbert space and let us consider $\varphi \in B_{fp} (\mathcal H)$. 

We shall check that every non-zero $\lambda \in \sigma (\varphi)$ satisfies the conditions of a Riesz point given in Definition \ref{d:Riesz-point-34937}. 

If $a_{\varphi} (x)$ is the annihilator polynomial of $\varphi$, from Proposition \ref{pro:spectru-finit-p3i3353} we have that $$a_{\varphi} (x) = (x - \lambda)^s \cdot p_\lambda (x)$$\noindent with $(x-\lambda, p_\lambda (x)) = 1$, and we can write $N(\lambda) = \Ker (\varphi - \lambda\text{Id})^s$ and $F(\lambda) = \Ker p_\lambda (\varphi)$.

It is clear that $N(\lambda)$ and $F(\lambda)$ are $\varphi$-invariant subspaces of $\mathcal H$.

Let $\mathcal H = W_{_\varphi} \oplus U_{_\varphi}$ be the AST-decomposition of $\mathcal H$ determined by $\varphi$. Since $N(\lambda) \subseteq W_{_\varphi}$, then $N(\lambda)$ is finite dimensional. Moreover, bearing in mind that $p_\lambda (\varphi)$ is a bounded operator, one has that $F(\lambda) = \Ker p_\lambda (\varphi)$ is a closed subspace of $\mathcal H$.

Finally, since $\varphi - \lambda \text{Id}$ is clearly nilpotent in $N (\lambda)$ and is invertible in $F(\lambda)$, it follows from the Bounded Inverse Theorem that $\varphi - \lambda \text{Id}$ is an homeomorphism of $F(\lambda)$, from where the proof is concluded.
\end{proof}

\begin{lem} \label{l:full-decomp-fii-pot-bound-383463} If $\mathcal H$ is a Hilbert space and $\varphi \in B_{fp} (\mathcal H)$, then the CN-decomposition $\varphi = \varphi_{_1} + \varphi_{_2}$  is a West decomposition of $\varphi$ (Definition \ref{def-full-decomopso3736}).
\end{lem}

\begin{proof} Since $\varphi_{_1}$ is compact and $\varphi_{_2}$ is nilpotent and, therefore, quasi-nilpotent, the CN-decomposition satisfies the conditions of Definition \ref{def-full-decomopso3736} and the claim is proved.
\end{proof}

\begin{thm} \label{th:Ris-Trac-Clas-boun-3843353} Every bounded finite potent endomorphism on a Hilbert space is a Riesz trace class operator.
\end{thm}

\begin{proof} Bearing in mind that $\varphi_{_1}$ is of trace class, the assertion is immediately deduced from Proposition \ref{prop:ev-boun-dp-isrisk-394373} and Lemma \ref{l:full-decomp-fii-pot-bound-383463}.
\end{proof}

\subsection{Trace and Determinant of a bounded finite potent endomorphism} \label{ss:tr-det-393473}

We shall now relate for bounded finite potent endomorphisms the Tate's trace of a finite potent endomorphism introduced in \cite{Ta} with the Leray trace defined in \cite{Ler} and with the trace of a Riesz trace class operator offered in \cite{El}. Given an arbitrary Hilbert space $\mathcal H$ and an endomorphism $\varphi \in B_{fp} (\mathcal H)$, we again denote by $\tr_{\mathcal H} (\varphi)$ the Trace's trace, by $\tr_{\mathcal H}^L (\varphi)$ the Leray trace and by $\tr^R_{\mathcal H} (\varphi)$ the trace of $\varphi$ as a Riesz trace class operator. Moreover we write $\tr ({\psi})$ to refer to the trace of a trace class operator $\psi$ and $\tr_E (f)$ to refer to the trace of an endomorphism $f$ on a finite-dimensional space $E$.

\begin{lem} \label{l:trace-woru3y353} Given a Hilbert space $\mathcal H$ and an endomorphism $\varphi \in B_{fp} (\mathcal H)$ with AST decomposition and $\mathcal H = W_{_\varphi} \oplus U_{_\varphi}$, then $$\tr_V (\varphi) = \tr_{W_{_\varphi}} (\varphi_{\vert_{_{W_{_\varphi}}}})\, .$$
\end{lem}

\begin{proof} According to \cite[page 150]{Ta}, $\tr_V (\varphi)$ can be computed as $$\tr_V (\varphi) = \tr_{W} (\varphi_{\vert_{_W}})\, ,$$ where $W$ is a finite dimensional linear subspace of  $V$, such that $W$ is $\varphi$-invariant and $\varphi^n (V) \subseteq W$ for a large $n\in \mathbb N$.

Bearing in mind that $W_{_\varphi}$ is $\varphi$-invariant and $W_{_\varphi} = \varphi^r (V)$ with $r= i(\varphi)$, the assertion is proved.
\end{proof}

\begin{lem} \label{l:trace-class-varp-1-2023636}  If $\mathcal H$ is a Hilbert space and $\varphi \in B_{fp} (\mathcal H)$ with CN-decomposition $\varphi = \varphi_{_1} + \varphi_{_2}$, then the Tate's trace $\tr_{\mathcal H} (\varphi)$ coincides with the trace  of $\varphi_{_1}$ as a trace class operator.
\end{lem}

\begin{proof} Since $\tr (\varphi)$ can be computed from the expression (\ref{eq:lid-tra-3937}), the statement is a direct consequence of Proposition \ref{pro:spectru-finit-p3i3353}, Corollary \ref{c:spectru,-varp-core-203453}, Lemma \ref{l:oinci-eig-alg-38363} and Lemma \ref{l:trace-woru3y353}
\end{proof}

\begin{prop} \label{p:trace-Riesc-Tat-393363} If $\mathcal H$ is a Hilbert space and $\varphi \in B_{fp} (\mathcal H)$, then $\tr_{\mathcal H} (\varphi) = \tr^R_{\mathcal H} (\varphi)$.
\end{prop}

\begin{proof} The claim follows immediately from Definition \ref{def:Riesz-trace-class-trace-292737}, Lemma \ref{l:full-decomp-fii-pot-bound-383463}, Theorem \ref{th:Ris-Trac-Clas-boun-3843353} and Lemma \ref{l:trace-class-varp-1-2023636}.
\end{proof}

Keeping the previous notation, if $\varphi \in B_{fp} (\mathcal H)$ and $\{\lambda_i (\varphi)\}_{i\in \{1, \dots, s\}}$ is the listing of all non-zero eigenvalues of $\varphi$, counted up to algebraic multiplicity, one has that \begin{equation} \label{eq:expli-tr-eig-38336} \tr_V (\varphi) = \sum_{i=1}^{s} \lambda_i (\varphi)\, .\end{equation}

However, for the computation of $\tr_V (\varphi)$ is not necessary to calculate the eigenvalues of $\varphi$ because from Lemma \ref{l:trace-woru3y353} we can compute $\tr_V (\varphi)$ from the matrix associated with $\varphi_{\vert_{_{W_{_\varphi}}}}$ in a Hamel basis of $W_{_\varphi}$.

\begin{prop} \label{p:Ler-tat-tr-39336} Given a Hilbert space $\mathcal H$ and a finite potent bounded endomorphism $\varphi \in B_{fp}$, one has that $\tr_{\mathcal H} (\varphi) = \tr^L_{\mathcal H} (\varphi)$.
\end{prop}

\begin{proof} With the notation of Section \ref{ss:Leray-tra-39373}, if $\mathcal H = W_{_\varphi} \oplus U_{_\varphi}$ is the AST-decomposition of $\varphi$, it is clear that $N(\varphi) = U_{_\varphi}$ and, since ${\mathcal H}/U_{_\varphi}$ is finite-dimensional, then the Leray trace $\tr^L_{\mathcal H} (\varphi)$ makes sense.

If $\tilde \varphi$ is the endomorphism of ${\mathcal H}/U_{_\varphi}$ induced by $\varphi$, fixing a linear isomorphism $\tau \colon W_{_\varphi} \overset {\sim} \longrightarrow {\mathcal H}/U_{_\varphi}$, from the commutative diagram of linear maps $$\xymatrix{W_{_\varphi}  \ar[rr]_{\sim}^{\tau} \ar[d]_{\varphi_{\vert_{_{W_{_\varphi}}}}} & &  {\mathcal H}/U_{_\varphi}  \ar[d]^{\tilde \varphi} \\ W_{_\varphi}  \ar[rr]_{\sim}^{\tau}  &  &  {\mathcal H}/U_{_\varphi} }\, ,$$ \noindent we deduce that $$\tr_{\mathcal H} (\varphi) = \tr_{W_{_\varphi}} (\varphi_{\vert_{_{W_{_\varphi}}}}) = \tr_{{\mathcal H}/U_{_\varphi}} ({\tilde \varphi}) = \tr^L_{\mathcal H} (\varphi)\, .$$
\end{proof}

We can now summarize the statements of Lemma \ref{l:trace-woru3y353}, Proposition \ref{p:trace-Riesc-Tat-393363} and Proposition \ref{p:Ler-tat-tr-39336} in the following 

\begin{thm} \label{th:eqi-trac-ler-ries-tat-39373} Given a Hilbert space $\mathcal H$, for every finite potent bounded endomorphism $\varphi \in B_{fp} (\mathcal H)$ with AST-decomposition $\mathcal H= W_{_\varphi} \oplus U_{_\varphi}$, one has that $$\tr_{\mathcal H} (\varphi) = \tr^R_{\mathcal H} (\varphi) = \tr^L_{\mathcal H} (\varphi) = \tr_{W_{_\varphi}} (\varphi_{\vert_{_{W_{_\varphi}}}})\, .$$
\end{thm}

To finish this section we shall study determinants of bounded finite potent endomorphisms.

If $V$ is an arbitrary $k$-space, let us now recall from \cite[Section 3.A]{HP} that a determinant for a finite potent endomorphism $\varphi \in \ed_k (V)$ can be defined from the following properties:
\begin{itemize}
\item if $V$ is finite dimensional, then $\Det^k_V(1 + \varphi)$
is the ordinary determinant;

\item if $W$ is a subspace of $V$ such that $\varphi W \subset W$,
then $$\Det^k_V(1 + \varphi) = \Det^k_W(1 + \varphi) \cdot
\Det^k_{V/W}(1 + \varphi)\, ;$$

\item if $\varphi$ is nilpotent, then $\Det^k_V(1 + \varphi) = 1$.
\end{itemize}

If $V= W_{_\varphi}\oplus U_{_\varphi}$ is the AST-decomposition of $V$ determined by $\varphi$, similar to Lemma \ref{l:trace-woru3y353}, one can check  that \begin{equation} \label{eq:expli-determin-finit-pot-39373} \Det^k_V(\text{Id} + \varphi) = \Det^k_{W_{_\varphi}} (\text{Id} + \varphi_{\vert_{_{W_{_\varphi}}}})\, ,  \end{equation}\noindent where $\Det^k_{W_{_\varphi}} (\text{Id} + \varphi_{\vert_{_{W_{_\varphi}}}}) $ is the determinant of the endomorphism $\text{Id} + \varphi_{\vert_{_{W_{_\varphi}}}}$ on the finite-dimensional vector space $W_{_\varphi}$. Moreover, if $\varphi = \varphi_{_1} + \varphi_{_2}$ is again the CN-decomposition of $\varphi$, it is clear that $$\Det^k_V(\text{Id} + \varphi) =  \Det^k_V(\text{Id} + \varphi_{_1})\, .$$

Let us again consider an arbitrary Hilbert space $\mathcal H$ and a bounded finite potent endomorphism $\varphi \in B_{fp} (\mathcal H)$. According to \cite[Proposition 3.11]{HP} one has that \begin{equation} \label{eq:widgk3uey538} \Det^{\mathbb C}_{\mathcal H} (\text{Id} + \varphi) = 1 + \sum_{r\geq 1} \tr_{\bigwedge^r {\mathcal H}} [\bigwedge^r \varphi]\, .\end{equation}

Hence,  one has that $\Det^{\mathbb C}_{\mathcal H} (\text{Id} + \varphi)$ genera\-lizes the determinant defined by B. Simon in \cite{Si} for trace
class operators $B$ on a separable Hilbert space from the formula:
$$\Det_1 (1 + \mu B) =1+ \sum_{n=1}^{\infty} \mu^n
\tr(\bigwedge^n(B)\,,$$\noindent where $\mu \in {\mathbb C}$.

Moreover, it follows from \cite[Proposition 3.18]{HP} that \begin{equation} \label{eq:wijnd7ydt353438} \Det^{\mathbb C}_{\mathcal H} (\text{Id} + \varphi) = \prod_{i=1}^{s} [1 + \lambda_i (\varphi) ]\, ,\end{equation} \noindent where $\{\lambda_i (\varphi)\}_{i\in \{1, \dots, s\}}$ is again the listing of all non-zero eigenvalues of $\varphi$, counted up to algebraic multiplicity. Readers can see that  expression (\ref{eq:wijnd7ydt353438}) shows that $\Det^{\mathbb C}_{\mathcal H} (\text{Id} + \varphi)$ also generalizes the definition of an infinite determinant for trace class operators offered by N. Dunford and J. Schwartz in \cite{DS}.

Accordingly,the expression (\ref{eq:expli-determin-finit-pot-39373}) allows us to offer an easy method for the calculation of classical infinite determinants in Functional Analysis for bounded finite potent endomorphisms.
\medskip

\section{Structure of the Adjoint of a Bounded Finite Potent Endomorphism} \label{s:adj-finite-potent-397}

This final section is devoted to characterizing the structure of the adjoint operator of a bounded finite potent endomorphism of a Hilbert space and to offer its main properties.

\begin{prop} \label{prop:adjoint-is-boun-fp-34837} If $\mathcal H$ is a Hilbert space and we consider $\varphi \in B_{fp} (\mathcal H)$, then the adjoint $\varphi^*$ is also a bounded finite potent endomorphism.
\end{prop}

\begin{proof} It is known that the adjoint of a bounded linear map of a Hilbert space is also bounded. Let ${\mathcal H} = W_{_\varphi} \oplus U_{_\varphi}$ be the AST-decomposition induced by $\varphi$. Since $W_{_\varphi}$ is finite-dimensional, then $W_{_\varphi}$ is a closed subspace of $\mathcal H$ and $\mathcal H = W_{_\varphi}\oplus W_{_\varphi}^\perp$.

If we now consider $v\in W_{_\varphi}^\perp$, one has that $$<w,\varphi^* (v)>_{\mathcal H} = <\varphi (w), v>_{\mathcal H} = 0$$\noindent for every $w\in W_{_\varphi}$ because $W_{_\varphi}$ is $\varphi$-invariant. Accordingly, $W_{_\varphi}^\perp$ is $\varphi^*$-invariant.

Moreover, assuming that $i(\varphi) = n$ and bearing in mind that $\varphi^n (h) \in W_{_\varphi}$ for all $h\in {\mathcal H}$, given $v\in W_{_\varphi}^\perp$, we have that $$<h, (\varphi^*)^n (v)>_{\mathcal H} = <\varphi^n (h), v>_{\mathcal H} = 0\, ,$$\noindent from where we deduce that $(\varphi^*)^n (v) \in {\mathcal H}^\perp = \{0\}$ and $(\varphi^*)_{\vert_{W_{_\varphi}^\perp}}$ is nilpotent.

Hence, $\text{Im } \, (\varphi^*)^n = (\varphi^*)^n (W_{_\varphi})$ and we conclude that $\varphi^*$ is finite potent.
\end{proof}

\begin{cor} \label{cor-dim-less-3946} If $\varphi \in B_{fp} (\mathcal H)$, ${\mathcal H} = W_{_{\varphi}} \oplus U_{_{\varphi}}$ is the AST-decomposition induced by $\varphi$ and ${\mathcal H} = W_{_{\varphi^*}} \oplus U_{_{\varphi^*}}$ is the AST-decomposition determined by $\varphi^*$, then $$\text{dim}_{\mathbb C} \, W_{_{\varphi^*}} \leq \text{dim}_{\mathbb C} \, W_{_{\varphi}}\, .$$
\end{cor}

\begin{proof} Bearing in mind that, from Lemma \ref{lem:inv-sub-AST-349346} and Proposition \ref{prop:adjoint-is-boun-fp-34837}, one has that $W_{_{\varphi}}^\perp \subseteq U_{_{\varphi^*}}$, then there exists a surjective linear map of finite-dimensional $\mathbb C$-vector spaces $${\mathcal H}/W_{_{\varphi}}^\perp \longrightarrow {\mathcal H}/U_{_{\varphi^*}} \to 0\, .$$

Thus, since ${\mathcal H}/U_{_{\varphi^*}} \simeq W_{_{\varphi^*}}$ and ${\mathcal H}/W_{_{\varphi}}^\perp \simeq W_{_{\varphi}}$ as $\mathbb C$-vector spaces, one obtains that $$\text{dim}_{\mathbb C} \, W_{_{\varphi^*}} \leq \text{dim}_{\mathbb C} \, W_{_{\varphi}}$$\noindent and the statement is proved.
\end{proof}

\begin{lem} \label{l:coin-dim-w-w-23346} If $\varphi \in B_{fp} (\mathcal H)$, ${\mathcal H} = W_{_{\varphi}} \oplus U_{_{\varphi}}$ is the AST-decomposition induced by $\varphi$ and ${\mathcal H} = W_{_{\varphi^*}} \oplus U_{_{\varphi^*}}$ is the AST-decomposition determined by $\varphi^*$, then $$\text{dim}_{\mathbb C} \, W_{_{\varphi^*}} = \text{dim}_{\mathbb C} \, W_{_{\varphi}}\, .$$
\end{lem}

\begin{proof} Since $\varphi^*$ is also bounded finite potent and it is clear that $(\varphi^*)^* = \varphi$, then it follows from Corollary \ref{cor-dim-less-3946} that $$\text{dim}_{\mathbb C} \, W_{_{\varphi}} = \text{dim}_{\mathbb C} \, W_{_{(\varphi^*)^*}} \leq \text{dim}_{\mathbb C} \, W_{_{\varphi^*}} \leq \text{dim}_{\mathbb C} \, W_{_{\varphi}}\, ,$$\noindent
from where the claim is deduced.
\end{proof}

\begin{prop} \label{pr:w-perp-is-eq-u-342352463} If $\varphi \in B_{fp} (\mathcal H)$, ${\mathcal H} = W_{_{\varphi}} \oplus U_{_{\varphi}}$ is the AST-decomposition induced by $\varphi$ and ${\mathcal H} = W_{_{\varphi^*}} \oplus U_{_{\varphi^*}}$ is the AST-decomposition determined by $\varphi^*$, then $U_{_{\varphi^*}} = W_{_{\varphi}}^\perp$.
\end{prop}

\begin{proof} If we consider the exact sequence of finite-dimensional $\mathbb C$-vector spaces $$0 \to U_{_{\varphi^*}}/W_{_{\varphi}}^\perp \longrightarrow {\mathcal H}/W_{_{\varphi}}^\perp \longrightarrow {\mathcal H}/U_{_{\varphi^*}} \to 0\, ,$$\noindent bearing in mind that ${\mathcal H}/U_{_{\varphi^*}} \simeq W_{_{\varphi^*}}$ and ${\mathcal H}/W_{_{\varphi}}^\perp \simeq W_{_{\varphi}}$ as $\mathbb C$-vector spaces, one has that $U_{_{\varphi^*}}/W_{_{\varphi}}^\perp = \{0\}$ and the assertion is checked.
\end{proof}

\begin{cor} \label{cor:ind-fp-ind-adju-236} If $\varphi \in B_{fp} (\mathcal H)$, then $i(\varphi) = i (\varphi^*)$.
\end{cor}

\begin{proof} Since Proposition \ref{pr:w-perp-is-eq-u-342352463} shows that $U_{_{\varphi^*}} = W_{_{\varphi}}^\perp$, it follows from the argumentation made in the proof of Proposition \ref{prop:adjoint-is-boun-fp-34837} that $i(\varphi^*) \leq i (\varphi)$. Thus, bearing in mind that $(\varphi^*)^* = \varphi$, we obtain that  $i(\varphi) = i (\varphi^*)$ because $$i(\varphi) =i( (\varphi^*)^*) \leq i(\varphi^*) \leq i (\varphi)\, .$$
\end{proof} 

\begin{prop} \label{pr:u-perp-is-eq-u-348363} If $\varphi \in B_{fp} (\mathcal H)$, ${\mathcal H} = W_{_{\varphi}} \oplus U_{_{\varphi}}$ is the AST-decomposition induced by $\varphi$ and ${\mathcal H} = W_{_{\varphi^*}} \oplus U_{_{\varphi^*}}$ is the AST-decomposition determined by $\varphi^*$, then $W_{_{\varphi^*}} = U_{_{\varphi}}^\perp$.
\end{prop}

\begin{proof} Since $U_{_{\varphi}} = U_{_{(\varphi^*)^*}} = (W_{_{\varphi^*}})^\perp$ and $W_{_{\varphi^*}}$ is closed, then $$W_{_{\varphi^*}} = (W_{_{\varphi^*}}^\perp)^\perp = U_{_{\varphi}}^\perp\, .$$
\end{proof}

A direct consequence of Proposition \ref{pr:w-perp-is-eq-u-342352463} and Proposition \ref{pr:u-perp-is-eq-u-348363} is

\begin{cor} \label{cor:fp-sum-dimu-perp-m-pr73} If $\mathcal H$ is a Hilbert space, $\varphi \in B_{fp} (\mathcal H)$ and ${\mathcal H} = W_{_{\varphi}} \oplus U_{_{\varphi}}$ is the AST-decomposition induced by $\varphi$, then $${\mathcal H} = W_{_{\varphi}}^\perp \oplus U_{_{\varphi}}^\perp\, .$$
\end{cor}

Bearing in mind that $W_{_{\varphi}}$ and $U_{_{\varphi}}$ are $\varphi^*$-invariant, Corollary \ref{cor:fp-sum-dimu-perp-m-pr73} shows that $\varphi^*$ can be computed from $(\varphi^*)_{\vert_{W_{_\varphi}^\perp}}$ and $(\varphi^*)_{\vert_{U_{_\varphi}^\perp}}$. Accordingly, if $h,h' \in {\mathcal H}$ such that $h = w + u$ and $h' = w' + u'$ with $w\in W_{_{\varphi}}$, $u \in U_{_{\varphi}}$, $w' \in U_{_{\varphi}}^\perp$ and $u' \in W_{_{\varphi}}^\perp$, one has that $$\begin{aligned} <\varphi (h), h'>_{\mathcal H} &= <\varphi (w), w'>_{\mathcal H} + <\varphi (u), u'>_{\mathcal H} = \\ &=  <w, \varphi^* (w')>_{\mathcal H} + <u, \varphi^* (u')>_{\mathcal H} = \\ &= <h, \varphi^* (h')>_{\mathcal H}\, .\end{aligned}$$

Moreover, from the above statements we also immediately prove that the adjoint of a bounded nilpotent endomorphism $\varphi$ is also a bounded nilpotent endomorphism. Accordingly, if $\varphi \in B_{fp} (\mathcal H)$ is a nilpotent endomorphism, then $W_{_{\varphi}} = \{0\}$ and it follows from Proposition \ref{pr:w-perp-is-eq-u-342352463} that $U_{_{\varphi^*}} = \mathcal H$, from where we deduce that $\varphi^*$ is nilpotent.

\begin{lem} \label{lem:core-nilp-part-nilp384363}  If $\mathcal H$ is a Hilbert space, $\varphi \in B_{fp} (\mathcal H)$ with CN-decomposition $\varphi = \varphi_{_1} + \varphi_{_2}$ and $\varphi ^*= (\varphi^*)_{_1} + (\varphi^*)_{_2}$ is the CN-decomposition of $\varphi^*$, then $(\varphi^*)_{_1} = (\varphi_{_1})^*$ and $(\varphi^*)_{_2} = (\varphi_{_2})^*$.
\end{lem}

\begin{proof} It follows from the properties of the adjoint operator that $\varphi^* = (\varphi_{_1})^* + (\varphi_{_2})^*$ and from Corollary \ref{cor:ind-fp-ind-adju-236} that $i((\varphi_{_1})^*) \leq 1$. 

Also, since $\varphi_{_1} \circ \varphi_{_2} = \varphi_{_2} \circ \varphi_{_1} = 0$, one has that $$(\varphi_{_1})^* \circ (\varphi_{_2})^* = (\varphi_{_2})^* \circ (\varphi_{_1})^* = 0\, .$$

Finally, bearing in mind that $(\varphi_{_1})^*$ and $(\varphi_{_2})^*$ are finite potent and $(\varphi_{_2})^*$ is nilpotent, the statement is a direct consequence of the uniqueness of the CN-decomposition of a finite potent endomorphism.
\end{proof}

We shall now study the spectrum of the adjoint operator $\varphi^*$.

\begin{prop} \label{p:spectrum-adjoint-ope-349363}  If $\mathcal H$ is a Hilbert space and $\varphi \in B_{fp} (\mathcal H)$, given a non-zero $\lambda \in \mathbb C$, one has that $\lambda \in \sigma (\varphi^*)$ if and only if ${\overline \lambda} \in \sigma (\varphi)$. Moreover, the algebraic multiplicity of a non-zero eigenvalue $\lambda$ of $\varphi^*$ coincides with the algebraic multiplicity of ${\overline \lambda}$ as an eigenvalue of $\varphi$.
\end{prop}

\begin{proof} To prove the statement, we only need to check that for every non-zero $\lambda \in \sigma (\varphi)$ one has that $$\text{dim}_{\mathbb C} \Ker (\varphi^* - {\overline \lambda} \text{Id})^n = \text{dim}_{\mathbb C} \Ker (\varphi - {\lambda} \text{Id})^n$$\noindent for every $n\in \mathbb N$.

Bearing in mind that $(\varphi^* - {\overline \lambda} \text{Id})^n = [(\varphi - {\lambda} \text{Id})^n]^*$, then $$\Ker (\varphi^* - {\overline \lambda} \text{Id})^n = [\text{Im}  (\varphi - {\lambda} \text{Id})^n]^\perp\, .$$

Moreover, if ${\mathcal H} = W_{_{\varphi}} \oplus U_{_{\varphi}}$ is again the AST-decomposition induced by $\varphi$, since $U_{_\varphi} \subseteq [\text{Im}  (\varphi - {\lambda} \text{Id})^n$ for every non-zero $\lambda \in \mathbb C$, and $W_{_\varphi}$ and $U_{_\varphi}$ are $(\varphi - {\lambda} \text{Id})^n$-invariant subspaces of $\mathcal H$, one has linear isomorphisms $$[\text{Im}  (\varphi - {\lambda} \text{Id})^n]^\perp \simeq {\mathcal H}/[\text{Im}  (\varphi - {\lambda} \text{Id})^n] \simeq W_{_\varphi}/[\text{Im}  (\varphi_{\vert_{W_{_\varphi}}} - {\lambda} \text{Id})^n]$$\noindent and, therefore, we deduce that $$\text{dim}_{\mathbb C} \Ker (\varphi^* - {\overline \lambda} \text{Id})^n = \text{dim}_{\mathbb C} \Ker (\varphi_{\vert_{W_{_\varphi}}} - {\lambda} \text{Id})^n = \text{dim}_{\mathbb C} \Ker (\varphi - {\lambda} \text{Id})^n$$\noindent for every $n\in \mathbb N$.
\end{proof}

A direct consequence of Proposition \ref{p:spectrum-adjoint-ope-349363} is

\begin{cor} \label{c:adjoin-spectrum-is-conj-38363} If $\mathcal H$ is a Hilbert space and $\varphi \in B_{fp} (\mathcal H)$, then $\sigma (\varphi^*) = {\overline {\sigma (\varphi)}}$.
\end{cor}

Furthermore, one has that

\begin{prop} \label{p:tra-det-adjoin-39363} Given a Hilbert space $\mathcal H$ and a bounded finite potent endomorphism $\varphi \in B_{fp} (\mathcal H)$, one has that:
\begin{itemize}

\item $\tr_{\mathcal H} (\varphi^*) = {\overline {\tr_{\mathcal H} (\varphi)}}$;

\item $\Det^{\mathbb C}_{\mathcal H} (1 + \varphi^*) = {\overline {\Det^{\mathbb C}_{\mathcal H} (1 + \varphi) }}$. 

\end{itemize}
\end{prop}

\begin{proof} Bearing in mind the expressions (\ref{eq:expli-tr-eig-38336}) and (\ref{eq:wijnd7ydt353438}), the assertions follows from Proposition \ref{p:spectrum-adjoint-ope-349363}.
\end{proof}

\begin{exam} \label{ex:boun-fp-adjo-38463} Let $\{u_j\}_{j\in \mathbb N}$ be an orthonormal basis of a separable Hilbert space $\mathcal H$. If we consider $\varphi \in B_{fp} (\mathcal H)$ determined by the conditions $$\varphi (u_j) = \left \{\begin{array}{ccl} (1+i) u_1 + u_2 + u_4 & \text{ if } & j = 1 \\ 2u_1 + (5-3i)  u_3 & \text{ if } & j = 2 \\ u_1 - 2u_2 + 3u_3 - 2u_4 & \text{ if } &  j = 3 \\ 0 & \text{ if } & j = 4 \\ \frac{1}{j^2} u_4 & \text{ if } & j \geq 5 \end{array} \right . \, ,$$\noindent an easy computation shows that $$\varphi^* (u_j) = \left \{\begin{array}{ccl} (1-i) u_1 + 2u_2 + u_3 & \text{ if } & j = 1 \\ u_1 - 2u_3 & \text{ if } & j = 2 \\ (5+3i) u_2 + 3u_3  & \text{ if } & j = 3 \\ u_1 - 2u_3 + \sum_{h\geq 5} \frac{1}{h^2} u_h & \text{ if } & j = 4 \\ 0  & \text{ if } & j \geq 5 \end{array} \right . \, ,$$

Thus, since $W_{_{\varphi}} = \langle u_1, u_2+u_4, u_3\rangle$ and $U_{_{\varphi}} = {\overline {\langle u_i \rangle_{i\geq 4}}}$, one has that:
\begin{itemize}

\item $W_{_{\varphi^*}} = U_{_{\varphi}}^\perp = \langle u_1, u_2, u_3 \rangle$;

\item $U_{_{\varphi^*}} = W_{_{\varphi}}^\perp = \langle u_2 - u_4\rangle \oplus {\overline {\langle u_j \rangle_{j\geq 5}}}$.
\end{itemize}

Also, it is clear that $i(\varphi) = i (\varphi^*) = 2$.

Moreover, since the explicit expressions of the core part and the nilpotent part of $\varphi$ are $$\varphi_{_1} (u_j) = \left \{\begin{array}{ccl} (1+i) u_1 + u_2 + u_4 & \text{ if } & j = 1 \\ 2u_1 + (5-3i) u_3 & \text{ if } & j = 2 \\ u_1 - 2u_2 + 3u_3 - 2u_4 & \text{ if } & j = 3 \\ 0 & \text{ if } & j \geq 4  \end{array} \right .$$\noindent and $$\varphi_{_2} (u_j) = \left \{\begin{array}{ccl} 0 & \text{ if } & j \leq 4 \\ \frac{1}{j^2} u_4 & \text{ if } & j \geq 5 \end{array} \right . \, ,$$\noindent with adjoint operators $$(\varphi_{_1})^* (u_j) = \left \{\begin{array}{ccl} (1-i) u_1 + 2u_2 + u_3 & \text{ if } & j = 1 \\ u_1 - 2u_3 & \text{ if } & j = 2 \\ (5+3i) u_2 + 3u_3  & \text{ if } & j = 3 \\ u_1 - 2u_3  & \text{ if } & j = 4 \\ 0  & \text{ if } & j \geq 5 \end{array} \right . \, ,$$\noindent and $$(\varphi_{_2})^* (u_j) = \left \{\begin{array}{ccl} 0  & \text{ if } & j \leq 3 \\  \sum_{h\geq 5} \frac{1}{h^2} u_h & \text{ if } & j = 4 \\ 0  & \text{ if } & j \geq 5 \end{array} \right . \, ,$$\noindent it is easy to check that these data are compatible with the statements of Lemma \ref{lem:core-nilp-part-nilp384363}.

Finally, bearing in mind that $$\varphi_{\vert_{W_{_\varphi}}} \equiv \begin{pmatrix} 1 + i & 2 & 1 \\ 1 & 0 & -2 \\ 0 & 5-3i & 3 \end{pmatrix} \, \text{ and } \, \varphi_{\vert_{W_{_{\varphi^*}}}} \equiv \begin{pmatrix} 1-i & 1 & 0 \\ 2 & 0 & 5+3i \\ 1 & -2 & 3 \end{pmatrix}\, $$\noindent in the bases $\{ u_1, u_2+u_4, u_3\}$ of $W_{_\varphi}$ and $\{ u_1, u_2, u_3\}$ of $W_{_{\varphi^*}}$ respectively, one has that $$\tr_{\mathcal H} (\varphi) = 4 + i \, ; \,  \tr_{\mathcal H} (\varphi^*) = 4 - i \, ; \Det_{\mathcal H} (\text{Id} + \varphi) = 15 + i  \, \text{ and } \Det_{\mathcal H} (\text{Id} + \varphi^*) = 15 - i  \, \, .$$
\end{exam}

\begin{rem} Given $\varphi \in B_{fp} (\mathcal H)$, although $W_{_\varphi}^\perp$ and $U_{_\varphi}^\perp$ are $\varphi^*$-invariant, we wish to point out that, in general, $W_{_\varphi}^\perp$ and $U_{_\varphi}^\perp$ are not $\varphi$-invariant. A counter-example for this fact is the bounded finite potent endomorphism studied in Example \ref{ex:boun-fp-adjo-38463}.
\end{rem}

\begin{rem}[Final Remark]  During the past few years the author of this work has extended several generalized inverses of finite square complex matrices to finite potent endomorphisms on infinite-dimensional inner product spaces in \cite{Pa-CN}, \cite{Pa-DMP} and \cite{Pa-Dr}. From the results of Section \ref{s:adj-finite-potent-397}, we hope, in forthcoming papers, to extend to bounded finite potent endomorphisms on arbitrary Hilbert spaces different ge\-neralized inverses of finite complex matrices that need the notion of the conjugate transpose matrix for their definitions.
\end{rem}

\medskip

{\bf Acknowledgment.-} The author would like to thank his colleague,  Dr. \'Angel A. Tocino Garc\'ia, in the Department of Mathematics of the University of Salamanca for his useful comments on properties of operators on Hilbert spaces.

\end{document}